\documentclass[smallextended]{svjour3}       % onecolumn (second format)
\smartqed  % flush right qed marks, e.g. at end of proof
\usepackage{graphicx}
\usepackage{amssymb,latexsym,amscd,amsmath,amsfonts,enumerate,supertabular}
\usepackage{graphicx}
\usepackage{tabularx}
\usepackage{subfigure}
 \usepackage{epstopdf}
 \usepackage{cite}
\begin{document}

\title{Exact finite difference schemes for three-dimensional linear systems with constant coefficients}

% \subtitle{Do you have a subtitle?\\ If so, write it here}

\titlerunning{Exact finite difference schemes for 3-D linear systems}        % if too long for running head

\author{Dang Quang A         \and
        Hoang Manh Tuan %etc.
}


\institute{Dang Quang A \at  Center for Informatics and Computing, Vietnam Academy of Science and Technology, 18 Hoang Quoc Viet, Cau Giay, Hanoi, Vietnam\\
%               first address \\
%               Tel.: +123-45-678910\\
%               Fax: +123-45-678910\\
              \email{dangquanga@cic.vast.vn}           %  \\
           \and
           Hoang Manh Tuan \at Institute of Information Technology, 
Vietnam Academy of Science and Technology, 18 Hoang Quoc Viet, Cau Giay, Hanoi, Vietnam\\
             \email{hmtuan01121990@gmail.com}
}

\date{Received: date / Accepted: date}

\maketitle

\begin{abstract}

In this paper implicit and explicit exact difference schemes (EDS) for system  $\textbf{x}' = A\textbf{x}$ of three linear differential equations with constant coefficients are constructed. Numerical simulations for stiff problem and for problems with periodic solutions on very large time interval demonstrate the efficiency and exactness of the EDS compared with high-order numerical methods. This result can be extended for constructing EDS for general  systems of  $n$ linear differential equations with constant coefficients and nonstandard finite difference (NSFD) schemes preserving stability properties for quasi-linear system of equations   $\textbf{x}' = A\textbf{x }+ f(\textbf{x})$.
\keywords{Exact finite-difference schemes \and Nonstandard finite-difference scheme \and Linear system \and Jordan form}
\subclass{MSC 65Q10 \and MSC 65L05 }
\end{abstract}

\section{Introduction}
The concepts of \textit{nonstandard finite difference schemes (NSFD) and exact finite difference schemes}  for differential equations were introduced by R. E. Mickens in $1980$ (see \cite{mickens1, mickens2, mickens4}). According to this the exact finite difference schemes are those NSFD, whose solution coincides with the exact solution of the differential equations at grid points. Nonstandard finite-difference (NSFD) schemes and exact finite difference schemes have become popular in recent years (see e.g. \cite{AL, dk1, Mickens3, Roeger2, Roeger3, Roeger4, Roeger5, Roeger6, Roeger7, Roeger8, Roeger9}) mainly because some methods are more efficient on preserving certain qualitative properties in the original differential equations or systems. A good review of NSFD methods can be found in \cite{mickens4}.\par

There are a lot of results of EDS for both ordinay and partial differential equations such as 
  \cite{Mickens3, Mickensp2,Roegerp1, Roeger9, Roegerp3, Lapinska, Zibaei}. Among them EDS for linear differential equations or system of differential equations with constant coefficients have attracted a special interest   \cite{Mickens3, Roeger3, Roeger4, Roegerp3}.\par

Recently in 2008,  Roeger  \cite{Roeger4} constructed exact difference schemes for the system of two differential equations with constant coefficients
\begin{equation}\label{a}
\textbf{x}'(t) = A\textbf{x}(t), \quad \textbf{x}(t) = \big(x(t), y(t)\big)^T, \qquad A \in M_{2 \times 2}(\mathbb{R})
\end{equation}
in the form
\begin{equation}\label{b}
\dfrac{\textbf{x}_{k + 1} - \textbf{x}_k}{\phi(h)} = A\big[\theta \textbf{x}_{k + 1} + (1 - \theta)\textbf{x}_k\big],
\end{equation}
where $\theta \in \mathbb{R}$ and $\phi(h) = h + \mathcal{O}(h^2), \quad h \to 0$.\par
The main idea of the construction is that instead of the system $\textbf{x}' = A\textbf{x}$ with a general $2 \times 2$ matrix $A$ the author considered the system $\textbf{u}' = J\textbf{u}$, where $J$ is a $2 \times 2$ Jordan canonical form. For each case of the Jordan forms $J$, the parameters $\theta$ and $\phi$ are determined so that  \eqref{b} is an exact difference scheme for the system $\textbf{u}' = J\textbf{u}$. Finally, due to the fact that $A$ is similar to $J$ the exact finite difference schemes for $\textbf{u}' = J\textbf{u}$ also are exact finite difference schemes for $\textbf{x}' = A\textbf{x}$. The obtained results show that $\theta$ and $\phi$ depend only on the eigenvalues of $A$.\par
 Before, Mickens \cite{mickens1, mickens2, mickens4} constructed  EDS for system of two linear differential equations with constant coefficients
\begin{equation*}
\dfrac{du}{dt} = au + bw, \quad \dfrac{dw}{dt} = {cu + dw},
\end{equation*}
in the form
\begin{equation*}
\dfrac{u_{k + 1} - \psi u_k}{\phi} = au_k + bw_k, \qquad \dfrac{w_{k + 1} - \psi w_k}{\phi} = cu_k + dw_k,
\end{equation*}
where
\begin{equation*}
\psi = \dfrac{\lambda_1e^{\lambda_2 h} - \lambda_2e^{\lambda_1 h}}{\lambda_1 - \lambda_2} = 1 + \mathcal{O}(h^2), \qquad \phi = \dfrac{e^{\lambda_1 h} - e^{\lambda_2 h}}{\lambda_1 - \lambda_2} = h + \mathcal{O}(h^2),
\end{equation*}
and  $\lambda_1$, $\lambda_2$ are the roots of the characteristic equation
\begin{equation*}
\det
\begin{pmatrix}
a - \lambda& b\\
c& d - \lambda 
\end{pmatrix}
= 0.
\end{equation*}
EDS proposed by Mickens and Roeger contain two parameters. However, different from the Roeger's scheme Mickens' scheme contains additional parameter  $\psi = 1 + \mathcal{O}(h^2)$ in discretization
of the first derivative while the right hand side is locally discretized.\par 
It is possible to say that system of linear differential equations is simple, so if there are available eigenvalues and Jordan structure of the coefficient matrix then we have an explicit formula for solution. But it should be emphasized that for EDS we have to compute parameters once at beginning, after that compute solution  recurrently, therefore the computational cost is cheaper than computing by formula of solution (including the exponential and trigonometric functions) with the guarantee of accuracy of the result.

Therefore,  in this paper we extend the Roeger's idea and Mickens's idea for constructing exact finite difference schemes for the system of three equations with constant coefficients 
\begin{equation}\label{de}
\textbf{x}'(t) = A\textbf{x}(t), \quad \textbf{x}(t) = \big(x(t), y(t), z(t)\big)^T, \quad t \in [0, T], \qquad A \in M_{3 \times 3}(\mathbb{R}).
\end{equation}
Differently from the Roeger's approximation of the first derivative with the use of one parameter $\phi (h)$ in denominator, here we add a parameter $\psi$ in numerator so that the number of parameters is equal to the number of differential equations in the system. Namely, instead of the difference scheme  \eqref{b} we use the difference scheme of the form
\begin{equation}\label{ds}
\dfrac{\textbf{x}_{k + 1} - \psi (h) \textbf{x}_k}{\phi(h)} = A\big[\theta \textbf{x}_{k + 1} + (1 - \theta)\textbf{x}_k\big],
\end{equation}
where $\psi (h) = 1 + \mathcal{O}(h^2)$.\par
Notice that the difference scheme \eqref{b} contains only two parameters $\phi$ and $\theta$, therefore it could be exact for two-dimensional system of equations. In some special cases it may be exact for three-dimensional system of equations but in general case of three-dimensional system  two parameters difference scheme \eqref{b} will not be exact. It is the reason why we introduce an additional parameter $\psi$ into the scheme \eqref{ds}.\par

In general, a system of  $n$ linear differential equations with constant coefficients has a fundamental system of solutions including  $n$ functions, therefore an EDS must contain at least $n$ parameters. The Roeger's exact scheme and  Mickens' scheme contain two parameters, therefore, they cannot ensure exactness for system of three linear equations. The addition of the parameter   $\psi$ into the Roeger's scheme is a simple extension of ours. \\
In general case, it is possible construct EDS for system of   
   $n$ equations based on Runge-Kutta methods. Namely, applying a s-stage Runge-Kutta method  \cite{Ascher, Hairer1, Hairer2} with coefficient matrix
  $A_{RK} = \big(a^*_{ij}\big)_{s \times s}$ and coefficients $b_{RK} = \big(b^*_1, b^*_2, \ldots , b^*_s \big)^T$ and $c = \big(c^*_1, c^*_2, \ldots , c^*_s\big)^T$ to the system $\textbf{x}' = A\textbf{x}$ we obtain the scheme
\begin{equation}\label{r1}
\dfrac{\textbf{x}_{k + 1} -  \textbf{x}_k}{h} = A\textbf{x}_k + \alpha_2 h A^2\textbf{x}_k + \alpha_3 h^2 A^3 \textbf{x}_k + \ldots + \alpha_s h^{s - 1} A^{s}\textbf{x}_k,
\end{equation}
where the coefficients 
 $\alpha_m = \alpha_m(a_{ij}^*, b_i^*), m = \overline{2, s}$ depend on $A_{RK}$ and $b_{RK}$. In the scheme \eqref{r1} replacing $h$ by the function $\phi(h) = h + \mathcal{O}(h^2)$ and adding the parameter $\psi(h) = 1 + \mathcal{O}(h^2)$ we obtain NSFD scheme for the system $\textbf{x}' = A\textbf{x}$
\begin{equation}\label{r2}
\dfrac{\textbf{x}_{k + 1} -  \psi(h)\textbf{x}_k}{\phi} =  A\textbf{x}_k + \alpha_2 \phi A^2\textbf{x}_k + \alpha_3 {\phi}^2 A^3 \textbf{x}_k + \ldots + \alpha_s \phi^{s - 1} A^{s }\textbf{x}_k.
\end{equation}
This scheme  \eqref{r2} contains $s + 1$ parameters. Therefore, it is possible construct EDS for the system $\textbf{x}' = A\textbf{x}$ with the dimension $n \leq s + 1$ from the scheme \eqref{r2}. Analogously, consider the following implicit difference schemes
\begin{equation}\label{r3}
\dfrac{\textbf{x}_{k + 1} -  \psi(h)\textbf{x}_k}{\phi(h)} =  A\textbf{x}_{k + 1} + \alpha_2 \phi A^2\textbf{x}_{k + 1} + \alpha_3 \phi^2 A^3 \textbf{x}_{k + 1} + \ldots + \alpha_s \phi^{s - 1} A^{s }\textbf{x}_{k + 1}.
\end{equation}
\begin{equation}\label{r4}
\dfrac{\textbf{x}_{k + 1} -  \psi(h)\textbf{x}_k}{\phi(h)} = A [\theta_1\textbf{x}_{k} + (1 - \theta_1)\textbf{x}_{k + 1}] + \sum_{m = 2}^{s}\alpha_m \phi^{m - 1}A^m[\theta_m \textbf{x}_{k} + (1 - \theta_m)\textbf{x}_{k + 1}].
\end{equation}
Depending on the dimension of the systems of equations we can choose the suitable number of parameters for the difference schemes to be exact. For example, in the case 
 $n = 2$ Roeger considers the scheme \eqref{r4} with $\psi = 1, \alpha_m = 0, m = \overline{2, s}$. In the case $n = 3$ under consideration we choose the scheme \eqref{ds} which is a particular case of the scheme \eqref{r4} with $\alpha_m = 0, m = \overline{2, s}$. It may be considered as a natural extension of the results of   Roeger and Mickens. Besides, we choose the explicit scheme  \eqref{r2} with $\alpha_m = 0, m = \overline{3, s}$, that is, the scheme of the form
\begin{equation}\label{r5}
\dfrac{\textbf{x}_{k + 1} -  \psi(h)\textbf{x}_k}{\phi} =  A\textbf{x}_k + \theta \phi A^2\textbf{x}_k.
\end{equation}
 In the  case of  $n \geq 3$ dimensions we can do in a similar way. In general, it is possible to construct EDSs based on the schemes of the form
  \eqref{r2}, \eqref{r3}, \eqref{r4} combined with the use of Jordan forms of matrices.\par

  In this work, we show that any three-dimensional linear system $\textbf{x}'(t) = A\textbf{x}(t)$, $\textbf{x}(t) = \big(x(t), y(t), z(t)\big)^T, A \in M_{3 \times 3}(\mathbb{R})$, has an exact finite-difference method in the forms \eqref{ds} and \eqref{r5}, where  $\psi$, $\phi$ and $\theta$ can be found explicitly in terms of the step-size $h$ and the eigenvalues $\lambda_{1, 2, 3}$ of the coefficient matrix $A$.  In Section $2$, we prove that if the parameters $\psi$, $\phi$ and $\theta$ are determined so that  \eqref{ds}/\eqref{r5}  is the exact difference scheme for $\textbf{u}' = J\textbf{u}$, where $J$ is $ 3\times 3$ Jordan form matrix then   \eqref{ds}/\eqref{r5} also will be exact for $\textbf{x}' = A\textbf{x}$ if $A$ is similar to $J$. Based on this fact, in Section $3$ and Section $4$ we construct implicit and explicit exact difference schemes for the system $\textbf{x}' = J\textbf{x}$. Next, in Section 5 we make a perturbation analysis for estimating the accuracy of EDS in the case of appearing of rounding errors due to the approximate computation of the parameters. In Section $6$ we report some numerical examples for stiff problems and problems with special properties on long time interval for demonstrating the efficiency and exactness of EDS in comparison with high-order numerical methods. 
  Some concluding remarks will be given in the last section.
\section{Why consider the system with Jordan form matrix?}
From Linear Algebra  it is well known that any $n \times n$ matrix $A$ is similar to a Jordan form matrix $J$. In the case $n = 3$ it is easy to list all  Jordan form matrices $J$ as stated in the following theorem (see e.g. \cite[Chapter 1]{SW}, \cite[Chapter 6]{KW}). 
\begin{theorem}\label{theorem1} 
Let $A$ be any $3 \times 3$ matrix. Then, $A$ is similar to one of the following Jordan form matrices $J$ depending on the set of its eigenvalues and the dimension of eigenspaces associated with the eigenvalues. Here, $\sigma(A)$ is the set of eigenvalues,  $\chi_A(t)$ and $m_A(t)$ are characteristic and minimal polynomials of  $A$, respectively.

\begin{table}
\caption{$3 \times 3$ Jordan form matrices}
\label{tabl1}       % Give a unique label
\begin{tabular}{lllll}
\hline\noalign{\smallskip}
$\sigma(A)$&$\chi_A(t)$& $m_A(t)$&$J$  \\
\noalign{\smallskip}\hline\noalign{\smallskip}
$\big\{\lambda_1, \lambda_2, \lambda_3\big\}$ & $(t - \lambda_1)(t - \lambda_2)(t - \lambda_3)$ &
$(t - \lambda_1)(t - \lambda_2)(t - \lambda_3)$ & $ \begin{pmatrix}
\lambda_1&0&0\\
0&\lambda_2&0\\
0&0&\lambda_3 
\end{pmatrix}$\\
%%%%%%%%%%%%%%%%%%%%%%%%%%%%%%%%%%%%%%%%%%%%%%%%%%%%%%%%
\\
$\big\{\lambda_1, \lambda_1, \lambda_2\big\}$ & $(t - \lambda_1)^2(t - \lambda_2)$ &
$(t - \lambda_1)(t - \lambda_2)$ & $\begin{pmatrix}
\lambda_1 & 0 &0\\
0&\lambda_1&0\\
0&0&\lambda_2
\end{pmatrix}$ \\
\\
%%%%%%%%%%%%%%%%%%%%%%%%%%%%%%%%%%%%%%%%%%%%%%%%%%%%%%%%
% $\big\{\lambda_1, \lambda_1, \lambda_2\big\}$ & $(t - \lambda_1)^2(t - \lambda_2)$ &
% $(t - \lambda_1)^2(t - \lambda_2)$ & $\begin{pmatrix}
% \lambda_1& 1&0\\
% 0&\lambda_1& 0\\
% 0&0&\lambda_2
% \end{pmatrix}$\\
% \\
%%%%%%%%%%%%%%%%%%%%%%%%%%%%%%%%%%%%%%%%%%%%%%%%%%%%%%%
$\big\{\lambda_1, \lambda_1, \lambda_2\big\}$ & $(t - \lambda_1)^2(t - \lambda_2)$ &
$(t - \lambda_1)^2(t - \lambda_2)$ & $\begin{pmatrix}
\lambda_1& 1&0\\
0&\lambda_1& 0\\
0&0&\lambda_2
\end{pmatrix}$\\
\\
% \hline
%%%%%%%%%%%%%%%%%%%%%%%%%%%%%%%%%%%%%%%%%%%%%%%%%%%%%%%%%%%%%%%%%%%%%%%%%%%%%%%%%%%%%%%%%%%%%%%%%%%%%%%%%%%%%%%%%
%&&&\\
$\big\{\lambda, \lambda, \lambda\big\}$ & $(t - \lambda)^3$ &
$(t - \lambda)$ & $\begin{pmatrix}
\lambda&0&0\\
0&\lambda&0\\
0&0&\lambda
\end{pmatrix}$\\
\\
%%%%%%%%%%%%%%%%%%%%%%%%%%%%%%%%%%%%%%%%%%%%%%%%%%%%%%%%%%%%%%%%%%%%%%%%%%%%%%%%%%%%%%%%%%%%%%%%%%%%%%%%%%%%%%%%%
$\big\{\lambda, \lambda, \lambda\big\}$ & $(t - \lambda)^3$ &
$(t - \lambda)^2$ & $\begin{pmatrix}
\lambda &1&0\\
0&\lambda &0\\
0&0&\lambda
\end{pmatrix}$\\
\\
%%%%%%%%%%%%%%%%%%%%%%%%%%%%%%%%%%%%%%%%%%%%%%%%%%%%%%%%%%%%%%%%%%%%%%%%%%%%%%%%%%%%%%%%%%%%%%%%%%%%%%%%%%%%%%%%%
$\big\{\lambda, \lambda, \lambda\big\}$ & $(t - \lambda)^3$ &
$(t - \lambda)^3$ & $\begin{pmatrix}
\lambda&1&0\\
0&\lambda&1\\
0&0&\lambda
\end{pmatrix}$\\
\noalign{\smallskip}\hline
\end{tabular}
\end{table}
\end{theorem}
Now, we consider the three-dimensional system of differential equations with constants coefficients 
 \eqref{de} and NSDF schemes of the form \eqref{ds}. Denote by $\mathcal{J}$ the set of all Jordan form $3 \times 3$ matrices.
\begin{theorem}\label{theorem2}
%Suppose that the parameters $\psi$, $\phi$, $\theta$  are determined so that  \eqref{ds} is
Suppose that the difference scheme 
\begin{equation}\label{dfu}
\dfrac{\textbf{u}_{k + 1} - \psi(h) \textbf{u}_k}{\phi(h)} = J\big[\theta \textbf{u}_{k + 1} + (1 - \theta)\textbf{u}_k\big]
\end{equation}
is exact for the system $\textbf{u}' = J\textbf{u}, J \in \mathcal{J}$. Then the difference scheme \eqref{ds} with the same  parameters $\psi$, $\phi$, $\theta$
is exact  for the system $\textbf{x}' = A\textbf{x}$, $A \in M_{3 \times 3}(\mathbb{R})$ if $A$ is similar to $J.$
\end{theorem}
\begin{proof}
Suppose that $A$ is similar to the Jordan form matrix $J \in \mathcal{J}$.  Then there exists a invertable matrix $P$ such that $ P^{-1}AP =J $. Making the transformations $\textbf{u} = P^{-1}\textbf{x}$ and $\textbf{u}_k = P^{-1}\textbf{x}_k$ we convert the system $\textbf{u}' = J\textbf{u}$ to $\textbf{x}' = A\textbf{x}$ and the difference scheme \eqref{dfu} to \eqref{ds}, respectively. Since \eqref{dfu} is exact for $\textbf{u}' = J\textbf{u}$, the difference scheme \eqref{ds} is exact for  $\textbf{x}' = A\textbf{x}$.
\end{proof}
The following theorem is a generalization of 
 Theorem \ref{theorem2} and is proved in a completely similar way.
 
\begin{theorem}\label{theorem2}
%Suppose that the parameters $\psi$, $\phi$, $\theta$  are determined so that  \eqref{ds} is
Suppose that the difference scheme \eqref{r2}/\eqref{r3}/\eqref{r4} is exact for the system $\textbf{u}' = J\textbf{u}, J \in \mathcal{J}$. Then the difference scheme \eqref{r2}/\eqref{r3}/\eqref{r4} is exact  for the system $\textbf{x}' = A\textbf{x}$, $A \in M_{n \times n}(\mathbb{R})$ if $A$ is similar to $J$.
\end{theorem}
From the above theorems we see that it suffices to construct exact difference schemes for three-dimensional systems with Jordan form matrices.
%%%%%%%%%%%%%%%%%%%%%%%%%%%%%%%%%%%%%%%%%%%%%%%%%%%%%%%%%%%%%%%%%%%%%%%%%%%%%%%%%%%%%%%%%%%%%%%%%%%%%%%%%%%%%%%%%
\section{Implicit exact difference schemes (IEDS) for $\textbf{x}' = J\textbf{x}$}
In this section we construct implicit EDS for the system $\textbf{x}' = A\textbf{x}$ in the form \eqref{ds}. To avoid the introduction of new variables we shall use $\textbf{x}$ instead of $\textbf{u}$ in the system of equations with Jordan form matrix. We construct exact difference schemes for this system and it suffices to replace $J$ by $A$ to obtain exact difference schemes for the system with the matrix $A$.
\subsection{The case when $A$ has 3 distinct eigenvalues}
In this case $A$ is similar to the Jordan form matrix
\begin{equation}\label{m1}
J_1 = 
\begin{pmatrix}
\lambda_1& 0& 0\\
0& \lambda_2& 0\\
0& 0& \lambda_3
\end{pmatrix}.
\end{equation}
First we consider the subcase $\lambda_1 \lambda_2 \lambda_3 \ne 0$. Then the linear system $\textbf{x}' = J_1 \textbf{x}$ has the exact solution
\begin{equation*}
x(t) = c_1e^{\lambda_1 t}, \qquad y(t) = c_2e^{\lambda_2 t}, \qquad z(t) = c_3e^{\lambda_3 t},
\end{equation*}
or equivalently
\begin{equation}\label{eq:7}
x_{k + 1} = x_k e^{\lambda_1 h}, \qquad y_{k + 1} = y_k e^{\lambda_2 h}, \qquad z_{k + 1} = z_k e^{\lambda_3 h}.
\end{equation}
Applying the difference scheme \eqref{ds} for the system $\textbf{x}' = J\textbf{x}$ we obtain
\begin{equation}\label{eq:8}
x_{k + 1} = \dfrac{\psi + \phi \lambda_1(1 - \theta)}{1 - \phi \lambda_1 \theta}x_k,\,
y_{k + 1} = \dfrac{\psi + \phi \lambda_2(1 - \theta)}{1 - \phi \lambda_2 \theta}y_k,\,
z_{k + 1} = \dfrac{\psi + \phi \lambda_3(1 - \theta)}{1 - \phi \lambda_3 \theta}z_k.
\end{equation}
Identifying \eqref{eq:8} and \eqref{eq:7} we come to the system for finding the parameters $\psi ,\phi , \theta $:
\begin{equation}\label{eq:9}
\begin{split}
\psi + \phi \lambda_1 (1 - \theta) = e^{\lambda_1 h}(1 - \phi \lambda_1 \theta),\\
\psi + \phi \lambda_2 (1 - \theta) = e^{\lambda_2 h}(1 - \phi \lambda_2 \theta),\\
\psi + \phi \lambda_3 (1 - \theta) = e^{\lambda_3 h}(1 - \phi \lambda_3 \theta).
\end{split}
\end{equation}
After some elementary transformations we can eliminate $\psi$ and obtain the system
\begin{equation}\label{eq:10}
\begin{split}
\phi(\lambda_1 - \lambda_2) + \phi\theta C_1= e^{\lambda_1 h} - e^{\lambda_2 h},\\
\phi(\lambda_2 - \lambda_3) + \phi\theta C_2 = e^{\lambda_2 h} - e^{\lambda_3 h},
\end{split}
\end{equation}
where for brevity we set
\begin{equation}\label{eq:11}
C_1 = \lambda_2 - \lambda_1 + \lambda_1e^{\lambda_1h} - \lambda_2e^{\lambda_2 h}, \qquad C_2 = \lambda_3 - \lambda_2 + \lambda_2e^{\lambda_2h} - \lambda_3e^{\lambda_3 h}.
\end{equation}
Dividing the first equation by the second one in \eqref{eq:10} we obtain the equation containing only one unknown $\theta$
\begin{equation*}
\dfrac{\lambda_1 - \lambda_2 + \theta C_1}{\lambda_2 - \lambda_3 + \theta C_2} = \dfrac{e^{\lambda_1 h} - e^{\lambda_2 h}}{e^{\lambda_2 h} - e^{\lambda_3 h}}.
\end{equation*}
The solution of this equation is 
\begin{equation}\label{eq:12}
\begin{split}
&\theta = \dfrac{T_1}{T_2}, \qquad T_1 = \lambda_1(e^{\lambda_2 h} - e^{\lambda_3 h}) +  \lambda_2(e^{\lambda_3h} - e^{\lambda_1 h}) +  \lambda_3(e^{\lambda_1h} - e^{\lambda_2 h}),\\
&T_2 = \lambda_1(1 - e^{\lambda_1 h})(e^{\lambda_2h} - e^{\lambda_3 h}) + \lambda_2(1 - e^{\lambda_2 h})(e^{\lambda_3h} - e^{\lambda_1 h}) + \lambda_3(1 - e^{\lambda_3 h})(e^{\lambda_1h} - e^{\lambda_2 h}).  
\end{split}
\end{equation}
After  $\theta$ is found, returning to \eqref{eq:10} and \eqref{eq:9} we obtain $\phi$ and $\psi$, namely 
\begin{equation}\label{eq:13}
\phi  = \dfrac{e^{\lambda_1h} - e^{\lambda_2h}}{\lambda_1 - \lambda_2 + \theta C_1}, \qquad \psi =  e^{\lambda_3h} - \phi \lambda_3 (e^{\lambda_3h}\theta + 1 - \theta), 
\end{equation}
where $C_1$ is given by \eqref{eq:11}. They are the parameters to be determined.\par
\begin{theorem}\label{theorem3}
The linear system \eqref{de} with the matrix of coefficients \eqref{m1} has an exact difference scheme of the form \eqref{ds}, where the parameters $\psi, \phi, \theta$ are determined by \eqref{eq:12} and \eqref{eq:13}.
\end{theorem}
%%%%%%%%%%%%%%%%%%%%%%%%%%%%%%%%%%%%%%%%%%%%%%%%%%%%%%%%%%%%%%%%%%%%%%%%%%%%%%%%%%%%%%%%%%%%%%%%%%%%%%%%%%%%%%%
Now we consider the special subcase $\lambda_1 \lambda_2 \lambda_3 = 0$. Without generality we suppose  $\lambda_1 = 0$. Then, $A$ is similar to the Jordan form matrix
\begin{equation}\label{m11}
J_1^* = 
\begin{pmatrix}
0& 0& 0\\
0& \lambda_2& 0\\
0& 0& \lambda_3
\end{pmatrix}.
\end{equation}
In this subcase, analogously as in the previous subcase, we find the parameters in the exact difference scheme
\begin{equation}\label{eq:14}
% \begin{split}
\psi = 1,\quad \phi = \dfrac{(\lambda_1 - \lambda_2)(e^{\lambda_1 h} - 1 )(e^{\lambda_2h} - 1)}{\lambda_1 \lambda_2 (e^{\lambda_1h} - e^{\lambda_2h})},\quad
\theta = \dfrac{\lambda_2(e^{\lambda_1 h} - 1) - \lambda_1(e^{\lambda_2h} - 1)}{(\lambda_1 - \lambda_2)(e^{\lambda_1h} - 1)(e^{\lambda_2h} - 1)}.
% \end{split}
\end{equation}
This result coincides with that of Roeger for the system of two linear equations with the coefficient matrix having two distinct eigenvalues
 $\lambda_1 \ne \lambda_2$ and both are nonzeros.
 
\begin{theorem}\label{theorem4}
The linear system \eqref{de} with the coefficient matrix \eqref{m11} has an exact difference scheme of the form \eqref{ds}, where the parameters  $\psi, \phi, \theta$ are given by \eqref{eq:14}.
\end{theorem}
Remark that, when the matrix $A$ has a pair of complex conjugate eigenvalues
\begin{equation*}
\lambda_{1, 2} = \alpha \pm \beta i,  \lambda_3 = \lambda, \beta \ne 0, \alpha, \beta, \lambda \in \mathbb{R},
\end{equation*}
following Theorem \ref{theorem3} we obtain the parameters $\psi, \phi, \theta$
\begin{equation}\label{eq:15}
\begin{split}
\theta &= \dfrac{T_1}{T_2}, \qquad \phi = \dfrac{2e^{\alpha h}\sin(\beta h)}{2\beta + T_3\theta}, \quad \psi = e^{\lambda h} - \phi \lambda (\theta e^{\lambda h} + 1 - \theta),\\
T_1 &= 2\beta(e^{\alpha h}\cos(\beta h) - e^{\lambda h}) + 2e^{\alpha h}(\lambda - \alpha)\sin(\beta h),\\
T_2 &= \alpha(1 - e^{\alpha h}\cos(\beta h))(-2e^{\alpha h}\sin(\beta h))+ \alpha e^{\alpha h}\sin(\beta h) (2e^{\lambda h} - 2e^{\alpha h} \cos(\beta h))\\ 
     &+ \beta(1 - e^{\alpha h}\cos(\beta h))(2 e^{\alpha h} \cos(\beta h) - 2e^{\lambda h}) + \beta e^{\alpha h}\sin(\beta h)(-2e^{\alpha h}\sin(\beta h))\\  
     &+ \lambda (1 - e^{\lambda h})(2 e^{\alpha h}\sin(\beta h)),\\
T_3 & = -2\beta + 2 \alpha e^{\alpha h}\sin{(\beta h)} + 2 \beta e^{\alpha h}\cos(\beta h).\\
\end{split}
\end{equation}
\begin{corollary}\label{corollary1}
The linear system \eqref{de} with the coefficient matrix 
\begin{equation}\label{m12}
J_1^{**} = 
\begin{pmatrix}
\alpha + \beta i& 0& 0\\
0& \alpha - \beta i& 0\\
0& 0& \lambda
\end{pmatrix},
\end{equation}
 has an exact difference scheme of the form \eqref{ds}, where the parameters $\psi, \phi, \theta$ are determined by \eqref{eq:15}.
\end{corollary}
%%%%%%%%%%%%%%%%%%%%%%%%%%%%%%%%%%%%%%%%%%%%%%%%%%%%%%%%%%%%%%%%%%%%%%%%%%%%%%%%%%%%%%%%%%%%%%%%%%%%%%%%%%%%%%%%
\subsection{The case $A$ has eigenvalues $\lambda_1 = \lambda_2 \ne \lambda_3$}
In this case the matrix
 $A$ is similar to one of the following Jordan form matrices
\begin{equation*}
J_2 = 
\begin{pmatrix}
\lambda_1& 0& 0\\
0& \lambda_1& 0\\
0& 0& \lambda_2
\end{pmatrix},
\qquad
J_3 = 
\begin{pmatrix}
\lambda_1& 1& 0\\
0& \lambda_1& 0\\
0& 0& \lambda_2
\end{pmatrix}.
\end{equation*}
\textbf{(i). When $A$ is similar to $ J_2 $}\\

 In this subcase the linear system  \eqref{de} with the coefficient matrix $J_2$ has exact solution
\begin{equation*}
x(t) = c_1e^{\lambda_1 t}, \qquad y(t) = c_2e^{\lambda_1 t}, \qquad z(t) = c_3e^{\lambda_2 t},
\end{equation*}
or equivalently
\begin{equation}\label{eq:16}
x_{k + 1} = x_ke^{\lambda_1 h}, \qquad y_{k + 1} = y_k e^{\lambda_1 h}, \qquad z_{k + 1} = z_k e^{\lambda_2 h}.
\end{equation}
Applying the difference scheme \eqref{ds} for $\textbf{x}' = J_2\textbf{x}$ we obtain
\begin{equation}\label{eq:17}
x_{k + 1} = \dfrac{\psi + \phi \lambda_1 (1 - \theta)}{1 - \phi \lambda_1 \theta}x_k, \quad y_{k + 1} = \dfrac{\psi + \phi \lambda_1 (1 - \theta)}{1 - \phi \lambda_1 \theta}y_k,\quad z_{k + 1} = \dfrac{\psi + \phi \lambda_2 (1 - \theta)}{1 - \phi \lambda_2 \theta}z_k.
\end{equation}
Identifying \eqref{eq:17} and \eqref{eq:16} we come to the system of conditions for determining
 $\psi, \phi, \theta$
\begin{equation}\label{eq:18}
% \begin{split}
\dfrac{\psi + \phi \lambda_1 (1 - \theta)}{1 - \phi \lambda_1 \theta} = e^{\lambda_1 h}, \qquad\dfrac{\psi + \phi \lambda_2 (1 - \theta)}{1 - \phi \lambda_2 \theta} = e^{\lambda_2 h}.
% \end{split}
\end{equation}
The above system \eqref{eq:18} is of two equations but contains three parameters. Therefore, it has infinite number of solutions. It is easy to find a dependence of $\psi, \phi$ on $ \theta $ as follows
\begin{equation}\label{eq:19}
\phi = \dfrac{e^{\lambda_1h} - e^{\lambda_2h}}{(\lambda_1 - \lambda_2)(1 - \theta) + \theta(\lambda_1e^{\lambda_1h} - \lambda_2e^{\lambda_2h})},\qquad \psi = e^{\lambda_1h}(1 - \phi \lambda_1 \theta) - \phi \lambda_1(1 - \theta).
\end{equation}
\begin{theorem}\label{theorem5}
The linear system \eqref{de} with the coefficient matrix
\begin{equation}\label{eq:20}
J_2 = \begin{pmatrix}
\lambda_1& 0 & 0\\
0& \lambda_1& 0\\
0& 0& \lambda_2
\end{pmatrix},
\end{equation}
has an exact difference scheme of the form \eqref{ds}, where the parameters $\psi, \phi, \theta$ satisfy the relations \eqref{eq:19}.
\end{theorem}
%%%%%%%%%%%%%%%%%%%%%%%%%%%%%%%%%%%%%%%%%%%%%%%%%%%%%%%%%%%%%%%%%%%%%
\textbf{(ii). When $A$ is similar to $J_3$}\\
In this case the system \eqref{de} with the coefficient matrix $J_3$ has an exact solution
\begin{equation*}
x(t) = (c_2t + c_1)e^{\lambda_1 t}, \qquad y(t) = c_2e^{\lambda_1 t}, \qquad z(t) = c_3e^{\lambda_2 t},
\end{equation*}
or equivalently
\begin{equation}\label{eq:20}
x_{k + 1} = (hy_k + x_k) e^{\lambda_1 h}, \qquad y_{k + 1} = y_k e^{\lambda_1 h}, \qquad z_{k + 1} = z_k e^{\lambda_2 h}.
\end{equation}
Applying the difference scheme \eqref{ds} to the system $\textbf{x}' = J_3\textbf{x}$ we obtain
\begin{equation}\label{eq:21}
\begin{pmatrix}
1 - \phi \lambda_1 \theta& -\phi \theta& 0\\
\\
0& 1 - \phi \lambda_1 \theta&0\\
\\
0& 0& 1 - \phi \lambda_2 \theta
\end{pmatrix}
\begin{pmatrix}
x_{k + 1}\\
\\
y_{k + 1}\\
\\
z_{k + 1}
\end{pmatrix}
=
\begin{pmatrix}
[\psi + \phi \lambda_1 (1 - \theta)]x_k + \phi(1 - \theta)y_k\\
\\
[\psi + \phi \lambda_1 (1 - \theta)]y_k\\
\\
[\psi + \phi \lambda_2 (1 - \theta)]z_k
\end{pmatrix}.
\end{equation}
Suppose the above system is not degenerate. It occurs if $(1 - \phi \lambda_1 \theta)(1 - \phi \lambda_2 \theta) \ne 0.$ Then the system has a unique solution

\begin{equation}\label{eq:22}
\begin{split}
&x_{k + 1} = \dfrac{\psi + \phi \lambda_1 (1 - \theta)}{1 - \phi \lambda_1 \theta}x_k + \phi \dfrac{1 - \theta + \psi \theta}{( 1- \phi \lambda_1 \theta)^2}y_k \\
&y_{k + 1} = \dfrac{\psi + \phi \lambda_1 (1 - \theta)}{1 - \phi \lambda_1 \theta}y_k,\\
&z_{k + 1} = \dfrac{\psi + \phi \lambda_2 (1 - \theta)}{1 - \phi \lambda_2 \theta}z_k.
\end{split}
\end{equation}
Identifying \eqref{eq:22} and \eqref{eq:20} we come to the system of conditions for determining the parameters $\psi$, $\phi$, $\theta$:
\begin{equation}\label{eq:23}
% \begin{split}
\dfrac{\psi + \phi \lambda_2 (1 - \theta)}{1 - \phi \lambda_2 \theta} = e^{\lambda_2 h},\quad
\dfrac{\psi + \phi \lambda_1 (1 - \theta)}{1 - \phi \lambda_1 \theta} = e^{\lambda_1 h},\quad
\phi \dfrac{1 - \theta + \psi \theta}{(1 - \phi \lambda_1 \theta)^2} = he^{\lambda_1 h}.
% \end{split}
\end{equation}
If $\lambda_1 = 0$ then  from the above relations we get
\begin{equation}\label{eq:p1}
\psi = 1, \qquad \phi = h, \qquad \theta = \dfrac{e^{\lambda_2 h} - \lambda_2 h - 1}{\lambda_2 h (e^{\lambda_2 h} - 1)}.
\end{equation}
%Trường hợp $\lambda_1 \lambda_2 \ne 0$ Solving the system of nonlinear equations \eqref{eq:23} is difficult. For treating the system we set  $\phi \theta = T$. 
Otherwise, setting $\phi \theta = T$ from the third equation of  \eqref{eq:23} we obtain a quadratic equation for $T$
\begin{equation}\label{eq:24}
\phi + (\psi - 1)T = (1 - \lambda_1T)^2he^{\lambda_1 h}.
\end{equation}
Subtracting $1$ from the first and second equations of  \eqref{eq:23} we obtain a system of two equations for $\psi - 1$ and $\phi$, where $T$ is considered as a parameter.
\begin{equation*}
\begin{pmatrix}
1& \lambda_2\\
\\
1& \lambda_1
\end{pmatrix}
\begin{pmatrix}
\psi - 1\\
\\
\phi
\end{pmatrix}
=
\begin{pmatrix}
(e^{\lambda_2h} - 1 )(1 - \lambda_2T)\\
\\
(e^{\lambda_1h} - 1 )(1 - \lambda_1T)
\end{pmatrix}.
\end{equation*}
Due to $\lambda_1 \ne \lambda_2$ the system has a solution
\begin{equation}\label{eq:25}
\begin{split}
&\psi - 1 = \dfrac{\lambda_1(e^{\lambda_2h} - 1) - \lambda_2(e^{\lambda_1h} - 1) + \lambda_1\lambda_2(e^{\lambda_1h} - e^{\lambda_2h})T}{\lambda_1 - \lambda_2},\\
&\phi = \dfrac{e^{\lambda_1 h} - e^{\lambda_2 h} + \big[\lambda_2(e^{\lambda_2h} - 1) - \lambda_1(e^{\lambda_1 h} - 1)\big]T}{\lambda_1 - \lambda_2}.
\end{split}
\end{equation}

Substituting \eqref{eq:25} into \eqref{eq:24} we obtain a quadratic equation for $T$
\begin{equation}\label{eq:26}
\begin{split}
\big[he^{\lambda_1h}\lambda_1^2 - \dfrac{\lambda_1 \lambda_2(e^{\lambda_1h} - e^{\lambda_2h})}{\lambda_1 - \lambda_2}\big]T^2
 &- \big[2he^{\lambda_1h}\lambda_1 + \dfrac{(\lambda_1 + \lambda_2)(e^{\lambda_2h} - e^{\lambda_1h})}{\lambda_1 - \lambda_2}\big]T\\
&+ \big[he^{\lambda_1h} - \dfrac{e^{\lambda_1h} - e^{\lambda_2h}}{\lambda_1 - \lambda_2}\big] = 0.
\end{split}
\end{equation}
The equation \eqref{eq:26} has the discriminant $\Delta = (e^{\lambda_1h} - e^{\lambda_2h})^2 > 0$, hence it has two distinct roots $T_1, T_2$ given by
\begin{equation}\label{eq:27}
T_1 = \dfrac{\big[2he^{\lambda_1h}\lambda_1 + \dfrac{(\lambda_1 + \lambda_2)(e^{\lambda_2h} - e^{\lambda_1h})}{\lambda_1 - \lambda_2}\big] - (e^{\lambda_1h} - e^{\lambda_2h})}{2[he^{\lambda_1h}\lambda_1^2 - \dfrac{\lambda_1 \lambda_2(e^{\lambda_1h} - e^{\lambda_2h})}{\lambda_1 - \lambda_2}\big]}.
\end{equation}
\begin{equation}\label{eq:28}
T_2 = \dfrac{\big[2he^{\lambda_1h}\lambda_1 + \dfrac{(\lambda_1 + \lambda_2)(e^{\lambda_2h} - e^{\lambda_1h})}{\lambda_1 - \lambda_2}\big] + (e^{\lambda_1h} - e^{\lambda_2h})}{2[he^{\lambda_1h}\lambda_1^2 - \dfrac{\lambda_1 \lambda_2(e^{\lambda_1h} - e^{\lambda_2h})}{\lambda_1 - \lambda_2}\big]}.
\end{equation}
After finding $T = \phi \theta$,  we can calculate $\psi$ and $\phi$ by the formulas \eqref{eq:25}, and  $\theta = T/\phi$. \par
The following proposition implies that only the value  $T = T_1$ is consistent with the assumption of non-degeneration of the system \eqref{eq:21}
\begin{proposition}\label{proposition1}
For $T_1$ and $T_2$ determined by \eqref{eq:26} and \eqref{eq:27} we have
\begin{equation*}
\lim_{h \to 0^+}\lambda_1 T_1 = 0, \qquad \lim_{h \to 0^+}\lambda_1 T_2 = 1.
\end{equation*}
\end{proposition}
\begin{proof}
The proposition is easily proved by the use of L'Hospital's rule for the indeterminate form of the type  $0/0$. 
\end{proof}

Thus, in the case of the Jordan form matrix $J_3$ the parameters of the difference scheme are determined as follows
\begin{equation}\label{eq:29}
\begin{split}
&\psi  = 1 + \dfrac{\lambda_1(e^{\lambda_2h} - 1) - \lambda_2(e^{\lambda_1h} - 1) + \lambda_1\lambda_2(e^{\lambda_1h} - e^{\lambda_2h})T_1}{\lambda_1 - \lambda_2},\\
&\phi = \dfrac{e^{\lambda_1 h} - e^{\lambda_2 h} + \big[\lambda_2(e^{\lambda_2h} - 1) - \lambda_1(e^{\lambda_1 h} - 1)\big]T_1}{\lambda_1 - \lambda_2}.\\
&\theta = \dfrac{T_1}{\phi}.
%\qquad T = \dfrac{\big[2he^{\lambda_1}\lambda_1 + \dfrac{(\lambda_1 + \lambda_2)(e^{\lambda_2h} - e^{\lambda_1h})}{\lambda_1 - \lambda_2}\big] - (e^{\lambda_1h} - e^{\lambda_2h})}{2[he^{\lambda_1h}\lambda_1^2 - \dfrac{\lambda_1 \lambda_2(e^{\lambda_1h} - e^{\lambda_2h})}{\lambda_1 - \lambda_2}\big]}.
\end{split}
\end{equation}
\begin{theorem}\label{theorem6}
The linear system \eqref{de} with the coefficient matrix 
\begin{equation}\label{eq:30}
J_3 = 
\begin{pmatrix}
\lambda_1& 1 & 0\\
0&\lambda_1 &0\\
0& 0& \lambda_2
\end{pmatrix}
\end{equation}
has an exact difference scheme of the form \eqref{ds}, where the parameters $\psi, \phi, \theta$ are given by \eqref{eq:29}, where $T_1$ is determined by \eqref{eq:27} if 
$\lambda_1 \neq 0$ and by \eqref{eq:p1} if $\lambda_1 = 0$.
\end{theorem}
%%%%%%%%%%%%%%%%%%%%%%%%%%%%%%%%%%%%%%%%%%%%%%%%%%%%%%%%%
\begin{remark}
The system of conditions \eqref{eq:23} is obtained from \eqref{eq:18} by adding the third equation. Therefore, the parameters satisfying \eqref{eq:23} also satisfies \eqref{eq:18}, i.e., the difference scheme with these parameters is exact for the linear system \eqref{de} having the matrix $A$ similar to $J_2$. Thus, in the case if $A$ has the set of eigenvalues  $\sigma(A) = \big\{\lambda_1, \lambda_1, \lambda_2 \big\}$ then the exact difference scheme may be determined by Theorem \ref{theorem6} not depending on the similar Jordan form matrices.
\end{remark}
%%%%%%%%%%%%%%%%%%%%%%%%%%%%%%%%%%%%%%%%%%%%%%%%%%%%%%%%%%%%%%%%%%%%%%%%%%%%%%%%%%%%%%%%%%%%%%%%%%%%%%%%%%%%%%%%
\subsection{The case $A$ has eigenvalues $\lambda_1 = \lambda_2 = \lambda_3 = \lambda$}
In this case $A$ is similar to one of the following three Jordan form matrices
\begin{equation*}
J_4 = 
\begin{pmatrix}
\lambda& 0& 0\\
0& \lambda& 0\\
0& 0& \lambda
\end{pmatrix}, \qquad
J_5 = 
\begin{pmatrix}
\lambda& 1& 0\\
0& \lambda& 0\\
0& 0& \lambda
\end{pmatrix}, \qquad
J_6 = 
\begin{pmatrix}
\lambda& 1& 0\\
0& \lambda& 1\\
0& 0& \lambda
\end{pmatrix}, \qquad
\end{equation*}
\textbf{(i). When $A$ is similar to $J_4$}\\
In this subcase the linear system with the coefficient matrix $J_4$ has the exact solution

\begin{equation*}
x(t) = c_1e^{\lambda t}, \qquad y(t) = c_2e^{\lambda t}, \qquad z(t) = c_3e^{\lambda t},
\end{equation*}
or equivalently
\begin{equation}\label{eq:31}
x_{k + 1} =x_ke^{\lambda h}, \qquad y_{k + 1} =y_ke^{\lambda h}, \qquad z_{k + 1} =z_ke^{\lambda h}.
\end{equation}
Applying the difference scheme \eqref{ds} to the system $\textbf{x}' = J_4\textbf{x}$ we obtain
\begin{equation}\label{eq:32}
x_{k + 1} = \dfrac{\psi + \phi \lambda (1 - \theta)}{1 - \phi \lambda \theta}x_k, \qquad y_{k + 1} = \dfrac{\psi + \phi \lambda (1 - \theta)}{1 - \phi \lambda \theta}y_k, \qquad z_{k + 1} = \dfrac{\psi + \phi \lambda (1 - \theta)}{1 - \phi \lambda \theta}z_k.
\end{equation}
Identifying \eqref{eq:32} and \eqref{eq:31} we obtain
\begin{equation}\label{eq:33}
\dfrac{\psi + \phi \lambda (1 - \theta)}{1 - \phi \lambda \theta} = e^{\lambda h}.
\end{equation}
\begin{remark}
The equation \eqref{eq:33}  contains three unknowns. So, it has infinitely number of solutions, namely, the set of its solutions is a two-dimensional linear space.  
 A simple solution of it is $\theta = 0, \psi = 1$ and $\phi = \dfrac{e^{\lambda h} - 1}{\lambda}$ ($\lambda \ne 0$).
\end{remark}
\begin{theorem}\label{theorem7}
The linear system \eqref{de} with the coefficient matrix
\begin{equation*}
J_4 = \begin{pmatrix}
\lambda& 0 & 0\\
0& \lambda& 0\\
0& 0& \lambda
\end{pmatrix}
\end{equation*}
has an exact difference scheme of the form\eqref{ds}, where the parameters $\psi, \phi, \theta$ satisfy \eqref{eq:33}.
\end{theorem}

\begin{flushleft}
\textbf{(ii). When $A$ is similar to $J_5$}\\
\end{flushleft}
In this subcase the linear system with the coefficient matrix $J_5$ has the exact solution
\begin{equation*}
x(t) = (c_1 + c_2t)e^{\lambda t}, \qquad y(t) = c_2e^{\lambda t}, \qquad z(t) = c_3e^{\lambda t},
\end{equation*}
or equivalently
\begin{equation}\label{eq:34}
x_{k + 1} =(x_k + y_kh)e^{\lambda h}, \qquad y_{k + 1} =y_ke^{\lambda h}, \qquad z_{k + 1} =z_ke^{\lambda h}.
\end{equation}
Applying the difference scheme \eqref{ds} to the system $\textbf{x}' = J_5\textbf{x}$ we obtain
\begin{equation}\label{eq:35}
\begin{split}
&x_{k + 1} = \dfrac{\psi + \phi \lambda (1 - \theta)}{1 - \phi \lambda \theta}x_k + \phi \dfrac{1 - \theta + \psi \theta}{(1 - \phi \lambda \theta)^2}y_k,\\
& y_{k + 1} = \dfrac{\psi + \phi \lambda (1 - \theta)}{1 - \phi \lambda \theta}y_k, \qquad z_{k + 1} = \dfrac{\psi + \phi \lambda (1 - \theta)}{1 - \phi \lambda \theta}z_k.
\end{split}
\end{equation}
Identifying \eqref{eq:35} and \eqref{eq:34}  we obtain the system of conditions for $\psi , \phi , \theta$
\begin{equation}\label{eq:36}
\dfrac{\psi + \phi \lambda (1 - \theta)}{1 - \phi \lambda \theta} = e^{\lambda h}, \qquad \phi \dfrac{1 - \theta + \psi \theta}{(1 - \phi \lambda \theta)^2} = he^{\lambda h}.
\end{equation}
\begin{remark}
The system \eqref{eq:36} is of two equations with three unknowns, so, it has infinitely many solutions, namely, its set of solutions is an one-dimensional linear space. A simple solution of it is $\theta = 0$,  $\phi(h) = h e^{\lambda h}$, $\psi = (1 - \lambda h)e^{\lambda h}.$ 
\end{remark}
\begin{theorem}\label{theorem8}
The linear system \eqref{de} with the coefficient matrix
\begin{equation}\label{eq:37}
J_5 = \begin{pmatrix}
\lambda& 1 & 0\\
0& \lambda& 0\\
0& 0& \lambda
\end{pmatrix},
\end{equation}
has an exact difference scheme of the form\eqref{ds}, where the parameters $\psi, \phi, \theta$ satisfy \eqref{eq:36}.
\end{theorem}
%%%%%%%%%%%%%%%%%%%%%%%%%%%%%%%%%%%%%%%%%%%%%%%%%%%%%%%%%%%%%%%%%%%%%%%%%%%%%%%%%%%%%%%%%%%%%%%%%%%%%%%%%%%%%%%%
\textbf{(iii). When $A$ is similar to $J_6$}\\
In this subcase \eqref{ds} with the coefficient matrix $J_6$ has the exact solution
\begin{equation*}
x(t) = (c_1 + c_2t + c_3\dfrac{t^2}{2})e^{\lambda t}, \qquad y(t) = (c_2 + c_3t)e^{\lambda t}, \qquad z(t) = c_3e^{\lambda t},
\end{equation*}
or equivalently
\begin{equation}\label{eq:38}
x_{k + 1} = (x_k + hy_k + \dfrac{h^2}{2}z_k) e^{\lambda h}, \qquad y_{k + 1} = (y_k + hz_k) e^{\lambda h}, \qquad z_{k + 1} = z_k e^{\lambda h}.
\end{equation}
Applying the difference scheme \eqref{ds} to the system $\textbf{x}' = J_6\textbf{x}$ we obtain
\begin{equation*}
\begin{pmatrix}
1 - \phi \lambda \theta& -\phi \theta& 0\\
\\
0& 1 - \phi \lambda \theta& -\phi \theta\\
\\
0& 0& 1 - \phi \lambda \theta
\end{pmatrix}
\begin{pmatrix}
x_{k + 1}\\
\\
y_{k + 1}\\
\\
z_{k + 1}
\end{pmatrix}
=
\begin{pmatrix}
[\psi + \phi \lambda (1 - \theta)]x_k + \phi(1 - \theta)y_k\\
\\
[\psi + \phi \lambda (1 -  \theta)]y_k + \phi(1 - \theta)z_k\\
\\
[\psi + \phi \lambda (1 -  \theta)]z_k
\end{pmatrix},
\end{equation*}
%Giả sử hệ phương trình đại số tuyến tính là không suy biến. Sử dụng phương pháp thế giải hệ ta thu được
Under the condition $\phi \lambda \theta \ne 1$ we have
\begin{equation}\label{eq:39}
\begin{split}
&z_{k + 1} = \dfrac{\psi + \phi \lambda (1 - \theta)}{1 - \phi \lambda \theta}z_k,\\
&y_{k + 1} = \dfrac{\psi + \phi \lambda (1 - \theta)}{1 - \phi \lambda \theta}y_k + \phi \dfrac{1 - \theta + \psi \theta}{( 1- \phi \lambda \theta)^2}z_k\\
&x_{k + 1} = \dfrac{\psi + \phi \lambda (1 - \theta)}{1 - \phi \lambda \theta}x_k + \phi \dfrac{1 - \theta + \psi \theta}{( 1- \phi \lambda \theta)^2}y_k + \phi \dfrac{1 - \theta + \psi \theta}{( 1- \phi \lambda \theta)^3} \phi \theta z_k.
\end{split}
\end{equation}

Identifying \eqref{eq:39} and \eqref{eq:38} we obtain the system of conditions for $\psi , \phi , \theta$
\begin{equation}\label{eq:40}
% \begin{split}
\dfrac{\psi + \phi \lambda (1 - \theta)}{1 - \phi \lambda \theta} = e^{\lambda h},\quad
\phi \dfrac{1 - \theta + \psi \theta}{(1 - \phi \lambda \theta)^2} = he^{\lambda h},\quad
\phi \dfrac{1 - \theta + \psi \theta}{(1 - \phi \lambda \theta)^3} \phi \theta = \dfrac{h^2}{2}e^{\lambda h}.
% \end{split}
\end{equation}

By some elementary calculations we have found 
\begin{equation}\label{eq:41}
\psi = \dfrac{e^{\lambda h}(2 - \lambda h)}{\lambda h + 2}, \qquad \phi = \dfrac{h(e^{\lambda h} + 1)}{\lambda h + 2}, \qquad \theta = \dfrac{1}{e^{\lambda h} + 1}.
\end{equation}
Clearly, $\phi \lambda \theta \ne 1$ if $h \ne -2/\lambda $.
Notice that when $\lambda_1 = \lambda_2 = \lambda_3 = 0$ we obtain
\begin{equation}\label{eq:42}
\psi = 1, \qquad \phi = h, \qquad \theta = \dfrac{1}{2}.
\end{equation}

\begin{theorem}\label{theorem9}
The linear system \eqref{de} with the coefficient matrix
\begin{equation}\label{eq:43}
J_6 = \begin{pmatrix}
\lambda& 1 & 0\\
0& \lambda& 1\\
0& 0& \lambda
\end{pmatrix},
\end{equation}
has an exact difference scheme of the form \eqref{ds}, where the parameters $\psi, \phi, \theta$ are determined by \eqref{eq:41}.
\end{theorem}

\begin{remark}
The equation \eqref{eq:33} for determining $\psi, \phi, \theta$ in the case of $J_4$ is the first equation in \eqref{eq:40}, and the system \eqref{eq:36} in the case of $J_5$ is the first two equations in \eqref{eq:40} in the case of $J_6$. Therefore, the parameter $\psi, \phi, \theta$ found from  \eqref{eq:40} are applicable for all cases of $J_4$, $J_5$ and $J_6$. It means that in the case if $A$ has the set of eigenvalues $\sigma(A) = \big\{\lambda, \lambda, \lambda \big\}$ then the exact difference scheme can be determined by Theorem \ref{theorem9}.
\end{remark}
\subsection{Summary of results}
The obtained above results of exact difference schemes for the three-dimensional linear system  \eqref{de} with constant coefficient matrix  $A$ can be summarized in Table \ref{tabl2} below
%%%%%%%%%%%%%%%%%%%%%%%%%%%%%%%%%%%%%%%%%%%%%%%%%%%%%

\begin{table}

\caption{The exact difference schemes for $\textbf{x}' = A\textbf{x}$ in dependence on the eigenvalues of $A$}
\label{tabl2}       

\begin{tabular}{llll}
\hline\noalign{\smallskip}
Case & $\sigma(A)$ & $\psi, \phi, \theta$ are determined by \\
\noalign{\smallskip}\hline\noalign{\smallskip}
$1$& $\big\{\lambda_1, \lambda_2, \lambda_3\big\}$  & Theorem \ref{theorem3}\\
Special case $1$& $ \big\{0, \lambda_2, \lambda_3\big\}$& Theorem \ref{theorem4}\\
Special case $2$& $ \big\{\alpha  + \beta i, \alpha  - \beta i,  \lambda \big\}$& Corollary \ref{corollary1}\\
$2$& $\big\{\lambda_1, \lambda_1,  \lambda_2 \big\}$& Theorem \ref{theorem6}\\
$3$& $ \big\{\lambda, \lambda, \lambda \big\}$& Theorem \ref{theorem9}\\
\noalign{\smallskip}\hline
\end{tabular}
\end{table}
%%%%%%%%%%%%%
Concerning the parameters $\psi$, $\phi$ and $\theta$ as functions of the stepsize of discretization it is easy to verify
\begin{lemma}
For the functions $\psi, \phi$ and $\theta$ determined by theorems in Table \ref{tabl1} we have\\
$(i). \lim_{h \to 0^+} \psi(\lambda_1, \lambda_2, \lambda_3) = 1,$\\
$(ii). \lim_{h \to 0^+} \phi(\lambda_1, \lambda_2, \lambda_3)/h = 1,$\\
$(iii). \lim_{h \to 0^+} \theta(\lambda_1, \lambda_2, \lambda_3) = \dfrac{1}{2}$.
\end{lemma}
\section{Explicit exact difference schemes (EEDS) for $\textbf{x}' = A\textbf{x}$}
In this section we construct explicit EDS for the system  $\textbf{x}' = A\textbf{x}$ in the form \eqref{r5}. From the previous section and   \cite{Roeger4} we conclude that there exist an EDS for systems with the matrix having a same set of eigenvalues, i.e. EDS does not depend on the Jordan structure of the matrix $A$. Specifically, for all matrices with a same set of eigenvalues, if a difference scheme is exact for a system with matrix having higher minimal polynomial then it is also the exact scheme for a system with matrix having lower minimal polynomials. Therefore, when constructing EDS for systems with matrices having a same set of eigenvalues it suffices to consider the case of matrix having highest minimal polynomial. It is why in this section instead of 6 cases of the matrix  $J$ as in Section 3 we consider only 3 cases of the matrix $J$ as follows.
\subsection{The case when $A$ has 3 distinct eigenvalues}
In this case matrix 
 $A$ is similar to $J_1$. Applying the difference scheme \eqref{r5} for the system $\textbf{x}' = J_1\textbf{x}$ we obtain
\begin{equation}\label{eq:8a}
\begin{split}
x_{k + 1} = \big(\psi + \phi \lambda_1 + \theta \phi^2 \lambda_1^2 \big)x_k,\\
y_{k + 1} = \big(\psi + \phi \lambda_2 + \theta \phi^2 \lambda_2^2\big)y_k,\\
z_{k + 1} = \big(\psi + \phi \lambda_3 + \theta \phi^2 \lambda_3^2 \big)z_k,\\
\end{split}
\end{equation}
Identifying \eqref{eq:8a} and \eqref{eq:7} we come to the system for finding the parameters $\phi, \psi, \theta$:
\begin{equation}\label{eq:9a}
\big(\psi + \phi \lambda_i + \theta \phi^2 \lambda_i^2 \big) = e^{\lambda_i h}, \qquad i = 1, 2, 3.
\end{equation}
%Hệ \eqref{eq:9a} đơn giản hơn với vế phải kiểu  đa thức theo $\phi \lambda_i$ đơn giản hơn so với hệ điều kiện \eqref{eq:9} với vế phải kiểu phân thức hữu tỷ. 
Suppose that $\lambda_1 + \lambda_2 \ne 0$. Then by consecutive eliminations it is easy to obtain the solution of the system  \eqref{eq:9a}:
\begin{equation}\label{eq:eeds1}
\begin{split}
\phi &= \dfrac{(\lambda_3^2 - \lambda_2^2)(\lambda_2^2e^{\lambda_1 h} - \lambda_1^2e^{\lambda_2 h}) - (\lambda_2^2 - \lambda_1^2)(\lambda_3^2e^{\lambda_2 h} - \lambda_2^2e^{\lambda_3 h})}{\lambda_1 \lambda_2 (\lambda_2 - \lambda_1)(\lambda_3^2 - \lambda_2^2) - \lambda_2 \lambda_3 (\lambda_3 - \lambda_2)(\lambda_2^2 - \lambda_1^2)},\\
\psi &= \dfrac{\lambda_2^2e^{\lambda_1h} - \lambda_1^2e^{\lambda_2h} - \lambda_1 \lambda_2 (\lambda_2 - \lambda_1)\phi}{\lambda_2^2 - \lambda_1^2}, \qquad \theta = \dfrac{e^{\lambda_3 h} - \psi - \lambda_3 \phi}{\lambda_3^2 \phi^2}.
\end{split}
\end{equation} 

\subsection{The case $A$ has eigenvalues $\lambda_1 = \lambda_2 \ne \lambda_3$}
In this case it suffices to consider the case when matrix 
 $A$ similar to $J_3$. Analoguously as above we obtain a system of conditions for determining the parameters   $\phi, \psi, \theta$:
\begin{equation}\label{eq:eeds2}
\phi + 2 \lambda_1 \theta \phi^2 = h e^{\lambda_1 h}, \qquad \psi + \phi \lambda_i + \theta \phi^2 \lambda_i^2 = e^{\lambda_i h}, \quad i = 1, 2.
\end{equation}
If $\lambda_1 = 0$ then from the system \eqref{eq:eeds2} we obtain
\begin{equation}\label{eq:eeds3}
\psi = 1, \qquad \phi = h, \qquad \theta = \dfrac{e^{\lambda_2 h} - \lambda_2 h - 1}{h^2 \lambda_2^2}.
\end{equation}
Otherwise, if $\lambda_1 \ne 0$ it is easy to get the solution of the system \eqref{eq:eeds2}
\begin{equation}\label{eq:eeds4}
\phi = \dfrac{(\lambda_2^2 h - \lambda_1^2 h + 2 \lambda_1)e^{\lambda_1 h} - 2 \lambda_1 e^{\lambda_2 h}}{(\lambda_1 - \lambda_2)^2}, \, \psi = \dfrac{(2 - \lambda_1 h)e^{\lambda_1 h} - \lambda_1 \phi}{2}, \, \theta = \dfrac{h e^{\lambda_1 h} - \phi}{2 \lambda_1\phi^2}.
\end{equation}

\subsection{The case $A$ has eigenvalues $\lambda_1 = \lambda_2 = \lambda_3 = \lambda$}
In this case it suffices to consider the case when the matrix 
 $A$ is similar to $J_6$. Then the system of conditions for the parameters $\phi, \psi, \theta$ is

\begin{equation}\label{eq:eeds4}
\psi + \phi \lambda + \theta \phi^2 \lambda_2^2 = e^{\lambda h}, \quad \phi + 2 \lambda \theta \phi^2 = h e^{\lambda h}, \quad \theta \phi^2 = \dfrac{h^2}{2}e^{\lambda h}.
\end{equation}
If $\lambda = 0$ from \eqref{eq:eeds4} we obtain $\psi = 1, \phi = h, \theta = \dfrac{1}{2}$. Otherwise, if $\lambda \ne 0$ the solution of \eqref{eq:eeds4} is
\begin{equation}\label{eq:eeds5}
\phi = (h - \lambda h^2)e^{\lambda h}, \qquad  \theta = \dfrac{h^2 e^{\lambda h}}{2 \phi^2}, \qquad \psi = e^{\lambda h} - \lambda \phi - \theta \phi^2 \lambda^2.
\end{equation}
\begin{remark}
From  Section 3 and Section 4 we see that the system of conditions for determining the parameters $\phi, \psi, \theta$ for implicit EDS is much more complicated than for explicit EDS. Specifically, the system of conditions for implicit EDS contains rational expressions while the system of conditions for  explicit EDS contains polynomial expressions.
\end{remark}
\section{Perturbation analysis}
Since the parameters of EDS contain exponential and trigonometric functions, in the process of computation rounding errors arise. Suppose that instead of the exact parameters $\psi, \phi, \theta$ we obtain only their approximate values  $\hat{\psi}, \hat{\phi}, \hat{\theta}$. Notice that the explicit EDS and implicit EDS for the system   $\textbf{x}' = A\textbf{x}$ can be written in the form $\textbf{x}_{k + 1} = Q(\psi, \phi, \theta)\textbf{x}_k := Q\textbf{x}_k$. Due to the fact that the iterative parameters are computed approximately instead of   $Q$ we only have $\hat{Q} :=(\hat{\psi}, \hat{\phi}, \hat{\theta}).$  Suppose 

\begin{equation*}
\hat{Q} = Q + \epsilon T, \qquad T  = (1)_{3 \times 3}.
\end{equation*}
Therefore, we obtain only the approximation $\hat{\textbf{x}_k}$ but not
   $\textbf{x}_k$ . It is easy to obtain the difference between $\hat{\textbf{x}_k}$ and $\textbf{x}_k$
\begin{equation*}
\hat{\textbf{x}_k} - \textbf{x}_k = (\hat{Q}^k - Q^k)\textbf{x}_0 = (Q^{k - 1} \epsilon T + Q^{k - 2} \epsilon^2 T^2 + \ldots + \epsilon^k T^k)\textbf{x}_0.
\end{equation*}
From here it follows
\begin{equation*}
||\hat{\textbf{x}_k} - \textbf{x}_k|| < C \sum_{i = 1}^k\epsilon^i, \qquad C = \Big(\max_{j = \overline{0, k-1}}||Q||^j||T||^{k - j}\Big)||\textbf{x}_0||.
\end{equation*}
We see that the error of computed solution mainly depends on the number of iterations and the rounding errors in computation of the parameters. When the number of iterations is large the error of
EDS may be large. However, this error slightly depends on  $h$, therefore we can overcome this phenomena as follows:  instead of computing $\textbf{x}_{k + 1}$  through  $\textbf{x}_{k}$ by the formula  $\textbf{x}_{k + 1} = Q(\psi(h), \phi(h), \theta(h))\textbf{x}_k$  with grid size $h = T/N$ we can compute $\textbf{x}_{k}$ through $\textbf{x}_0$ with grid size $h^* = t_k$, i.e. use the formula  $\textbf{x}_{k} = Q(\psi(h^*), \phi(h^*), \theta(h^*)\textbf{x}_0$. Equivalently, instead of solving the problem on interval $[0, t_k]$ with $h = T/N$ via a large number of steps we can compute $x_k$ via $x_0$ with the step $h^* = t_k$. Thus, after only one iteration we obtain  $\textbf{x}_{k}$. This is the advantage of EDS compared with high-order numerical methods because for ensuring the accuracy these methods must use small grid sizes. From the numerical examples in the next section it will be seen that even in the presence of rounding errors EDS are more efficient than high-order numerical methods.

 \section{Numerical simulations}
In this Section we perform some numerical simulations for confirming the validity of theoretical results obtained in the previous sections. The numerical simulations for a stiff problem and problems with specific properties demonstrate the advantage of EDS over high-order numerical methods.
 \begin{example}
Consider the system \eqref{de} with the coefficient matrix
\begin{equation*}
A = 
\begin{pmatrix}
21& -8& -19\\
18& -7& -15\\
16& -6& -15
\end{pmatrix}.
\end{equation*}
\end{example}
The set of eigenvalues of $A$ is $\sigma(A) = \{-1, \pm i\}$. For the initial conditions $x(0) = 0, y(0) = -50, z(0) = 50$, the system has the exact solution
 \begin{equation*}
 \begin{split}
x(t) = 100e^{-t} - 100\cos(t) - 450\sin(t),\\
 y(t) = 150\cos(t) - 200e^{-t} - 600\sin(t),\\
z(t) = 200e^{-t} - 150\cos(t) - 250\sin(t).
 \end{split}
 \end{equation*}
% The implicit  and explicit difference schemes for the system ara determined by Theorem  \ref{theorem4} and . 
The implicit exact difference scheme for the system are determined from Theorem \ref{theorem4}.  The exact solution of the system and the the solution of the exact difference schemes are depicted in Figures \ref{fig:1}.
 \begin{figure}
  \includegraphics[width=1.11\textwidth]{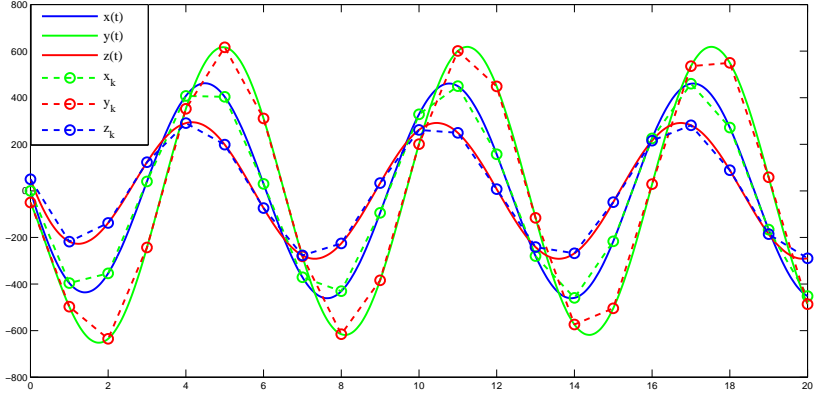}
\caption{Exact solutions and Implicit exact difference scheme }
\label{fig:1}      
\end{figure}
\begin{example}
Consider the system \eqref{de} with the coefficient matrix
\begin{equation*}
A = 
\begin{pmatrix}
3& -1& -3\\
-6& 2& 6\\
6& -2& -6
\end{pmatrix}.
\end{equation*}
\end{example}
The set of eigenvalues of $A$ is $\sigma(A) = \{0, -1\}$. For the initial conditions $x(0) = 0, y(0) = -40, z(0) = 50$, the system has the exact solution
 \begin{equation*}
% \begin{split}
x(t) =110e^{-t} - 110,\quad y(t) = 180 - 220e^{-t}, \quad z(t) = 220e^{-t} - 170.
% \end{split}
 \end{equation*}
% The implicit  and explicit difference schemes for the system ara determined by Theorem  \ref{theorem4} and . 
The explicit exact difference schemes for the system are determined by Subsection 4.2.  The exact solution of the system and the the solution of the exact difference schemes are depicted in Figures \ref{fig:2}.\par
 \begin{figure}
  \includegraphics[width=1.11\textwidth]{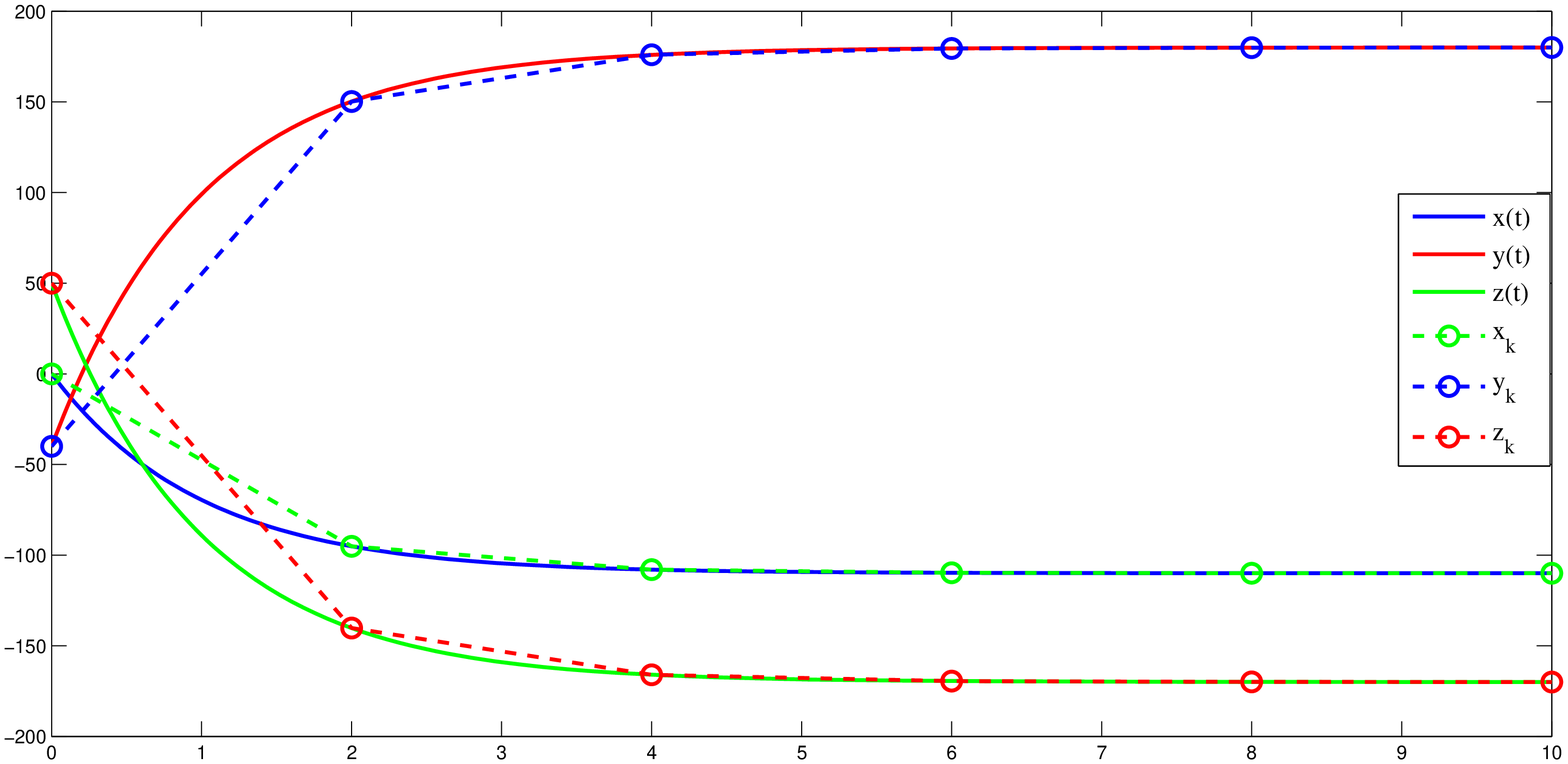}
\caption{Exact solutions and Explicit exact difference scheme}
\label{fig:2}      
\end{figure}
From the two above numerical examples we see that the solution of the constructed EDS \textbf{almost coincide}  with the  exact solution of the system of differential equations at grid points due to the insignificant rounding errors in computing the parameters of EDS. In the ideal case when the rounding errors are absent the solution of EDS must coincide with the solution of the system of differential equations for any grid size $h$. 
\begin{example}
Consider the system $\textbf{x}' = A\textbf{x}$, $t \in [0, T]$ with the coefficient matrix
\begin{equation*}
A = 
\begin{pmatrix}
0& -1& 0\\
1& 0& 0\\
0&0 & \lambda
\end{pmatrix},
\end{equation*}
\end{example}
where $\lambda > 0$. The set of eigenvalues of $A$ is $\sigma(A) = \{\pm i,  \lambda\}$. For the initial conditions $x(0) = 1, y(0) = 0, z(0) = 1$ and $\lambda = 1$, the system has the exact solution $x(t) = \cos(t)$, $y(t) = \sin(t)$, $z(0) = e^{\lambda t}$. The components $x(t)$, $y(t)$ are periodic and 
$x^2(t) + y^2(t) = 1$ for all $t \in [0, T]$. It is easy to proved that (see \cite[Problem 1, Section 3.9]{Ascher}):
\begin{enumerate}
\item The solution  $(x_k, y_k)$ obtained by the explicit Euler method  spirals out.
\item The solution $(x_k, y_k)$  obtained by the implicit Euler method  spirals in.
\item The solution $(x_k, y_k)$  obtained by the trapezoidal method  forms an approximate
circle as desired.
\end{enumerate}
In general, many numerical methods of higher order of accuracy such as Runge-Kutta or Taylor methods cannot preserve invariant properties of the differential problems although their global errors tend to zero as $h \rightarrow 0 $. It means that the choice of small grid size only ensures  the accuracy of the methods but not ensure the invariant properties of the problems. For the methods of higher order of accuracy including one step methods and multistep methods when solving the problems numerically  it is necessary to take grid size small because all convergence theorems are stated for $h \rightarrow 0 $. Therefore, when the final time  $T >>1$ it is impossible to choose grid size very small because the number of steps of computation becomes extremely large and this may cause difficulty with computer memory and computation time. Moreover, when $h$ is very small the accuracy may decrease due to rounding errors.\par

Now we compare the accuracy of EDS with some typical higher order methods such as Runge-Kutta and Taylor methods \cite{Ascher, Hairer1, Hairer2}.

In this example $\lambda > 0$, therefore the problem is unstable and it is not necessary to use implicit A-stable or L-stable Runge-Kutta methods. Moreover, since the system is linear, implicit schemes are easily reduced to explicit ones, therefore, we shall use the classical four-stage Runge-Kutta method. Besides, we consider the Taylor method of $5$ order of accuracy

%nên bài toán là không ổn định, vì thế không nhất thiết phải sử dụng các phương pháp Runge-Kutta ẩn ổn định A hoặc ổn định L. Thêm vào đó hệ phương trình là tuyến tính nên lược đồ dạng ẩn dễ dàng được đưa về lược đồ dạng hiển, vì thế chúng tôi chỉ sử dụng phương pháp Runge-Kutta bốn nấc kinh điển.

% Song song với đó chúng tôi xét thêm phương pháp Taylor chính xác cấp $5$, trong ví dụ này dễ dàng xây dựng các phương pháp Taylor có cấp chính xác cao do phương trình là tuyến tính.\par

Suppose, we need to find the approximate value for the exact solution at the time $t = T$. The comparison of errors of the methods are given in Table \ref{tabl1}, where  
 $error = |x_N - x(t_N)| + |y_N - y(t_N)| +  |z_N - z(t_N)|$ is used as a measure of accuracy of methods, 
  $x_N, y_N, z_N$ are computed solution  by methods with the grid size   
   $h = T/N$, $x(t_N), y(t_N), z(t_N)$ are values of the exact solution at $t_N = Nh = T$. 
For methods of higher order of accuracy we take small grid size for guaranteeing accuracy and convergence, but for EDS it is not needed. Hence, we use the grid size $h = T$ for avoiding the 
decrease of accuracy after a large number of iterations due to rounding errors.\par 
\begin{table}
\caption{Error of the methods}
\label{tabl1}       % Give a unique label
\begin{tabular}{lllllll}
\hline\noalign{\smallskip}
$T, \lambda$&$h$&IEDS error&EEDS error&RK4 error&Taylor error&Trapezoidal error\\
\hline\noalign{\smallskip}
$1, 1$&$10^{-5}$&3.2618e-011&3.8608e-011&3.1419e-014&3.4646e-011&1.0599e-010\\
&$10^{-4}$&1.5561e-012&1.2415e-012&1.2990e-014&1.5561e-012&3.4218e-009\\
&$10^{-3}$&4.9460e-013&8.1424e-013&3.3751e-014&1.2468e-013&3.4167e-007\\
&$10^{-2}$&4.5852e-014&4.6851e-014&3.4000e-010&2.9421e-014&3.4167e-005\\
&$10^{-1}$&3.2196e-015&7.7716e-015&3.2526e-006&7.7251e-010&0.0034\\
&$1$&7.7716e-016&1.1102e-016&0.0195&4.4648e-004&0.3829\\
\hline\noalign{\smallskip}
$10, 10^{-1}$&$h = 10^{-5}$& 1.0909e-010&4.9326e-011&1.3178e-013&1.0909e-010&2.1743e-010\\
&$h = 10^{-4}$&3.0020e-011&4.3151e-011&8.3267e-015&3.7715e-011&1.1621e-008\\
&$h = 10^{-3}$&1.3045e-012&1.0316e-012&1.3023e-013&1.3045e-012&1.1548e-006\\
&$h = 10^{-2}$&7.4307e-013&1.6542e-013&1.1505e-009&1.3267e-013&1.1548e-004\\
&$h = 10^{-1}$&3.7637e-014&3.7970e-014&1.1280e-005&2.6824e-009&0.0115\\
&$h = 1$&2.9976e-015&3.4417e-015&0.0995&0.0025&0.8470\\
&$h = 10$&1.3323e-015&1.4433e-015&524.6383&1.6976e+003&1.2944\\
\hline\noalign{\smallskip}
$10^{2}, 10^{-2}$&$h = 10^{-4}$&1.3930e-010&1.0862e-010& 8.1490e-014&1.3930e-010&1.1415e-007\\
&$h = 10^{-3}$&3.5170e-011&4.5318e-011&1.1456e-012&3.5170e-011&1.1406e-005\\
&$h = 10^{-2}$&2.5135e-012&2.1225e-012&1.1430e-008&1.3967e-012&0.0011\\
&$h = 10^{-1}$&1.9151e-013&7.0943e-013&1.1612e-004&2.7668e-008&0.1150\\
&$h = 1$& 3.3640e-014&6.0507e-014&0.6364&0.0230&1.3021\\
&$h = 10$&1.7208e-014&9.6589e-015&1.4626e+026&1.0099e+031&2.7849\\
&$h = 10^{2}$&1.1102e-016&3.3307e-016&4.3282e+006&1.4679e+009&2.6896\\
\hline\noalign{\smallskip}
$10^{3}, 10^{-3}$&$h = 10^{-4}$&1.1436e-010&6.5578e-009&1.7082& 4.8609e-009&1.1592e-006\\
&$h = 10^{-3}$&1.1436e-010&7.7965e-010&1.1645e-011&1.1436e-010&1.1577e-004\\
&$h = 10^{-2}$&3.9845e-011&5.0336e-011&1.1595e-007&3.6106e-011&0.0116\\
&$h = 10^{-1}$&2.0014e-012&7.6292e-012&0.0012&2.7916e-007&1.1135\\
&$h = 1$&7.6816e-013&7.4385e-014&1.3873&0.2445&2.7559\\
&$h = 10$&1.0358e-013&6.9056e-014&2.0312e+260&Inf&2.5752\\
&$h = 10^{2}$&3.2196e-015&4.4409e-015&2.0587e+066&3.6683e+091&1.5772\\
&$h = 10^{3}$&5.5511e-016&4.4409e-016&4.1833e+010&1.3972e+015&2.6670\\
\hline\noalign{\smallskip}
%$10^{4}$&$h = 10^{-4}$&Out of memory&Out of memory&Out of memory&Out of memory&Out of memory\\
$10^4, 10^{-4}$&$h = 10^{-3}$&2.1401e-009&4.7583e-009&1.6131&4.5953e-009&0.0010\\
&$h = 10^{-2}$&1.5582e-010&1.6500e-010&1.0436e-006&1.2046e-010&0.1024\\
&$h = 10^{-1}$&3.3033e-011&3.9898e-011&0.0100&2.3738e-006&2.4012\\
&$h = 1$&6.3882e-012&1.0866e-011&1.2578&5.2064&2.0189\\
&$h = 10$&7.4413e-013&1.0292e-013&NaN&NaN&2.6229\\
&$h = 10^{2}$&3.3529e-014&5.6566e-014&NaN&NaN&1.3602\\
&$h = 10^{3}$&6.0507e-015&3.4417e-015&1.6389e+106&2.8260e+151&2.2193\\
&$h = 10^{4}$&1.6653e-016&1.1102e-016&4.1683e+014&1.3897e+021& 0.6356\\
\hline\noalign{\smallskip}
% $10^{5}$&$h = 10^{-4}$&Out of memory&Out of memory&Out of memory&Out of memory&Out of memory\\
% &$h = 10^{-3}$&Out of memory&Out of memory&Out of memory&Out of memory&Out of memory\\
$10^5, 10^{-5}$&$h = 10^{-2}$&8.3200e-009& 2.8834e-009&8.6926e-006&4.6617e-009& 1.0818\\
&$h = 10^{-1}$&2.3169e-010&9.1972e-011&0.0952& 2.2129e-005&1.9557\\
&$h = 1$&3.2853e-011&4.3130e-011&1.0351&3.8828e+006&1.0811\\
&$h = 10$&2.0601e-011&7.6230e-012&NaN&NaN&1.4332\\
&$h = 10^{2}$&3.2153e-013&2.0207e-013&NaN&NaN&1.1208\\
&$h = 10^{3}$&5.5303e-014&5.1750e-014&NaN&NaN&2.3456\\
&$h = 10^{4}$&3.4431e-014&5.6760e-015&1.5835e+146&2.6870e+211&2.0414\\
&$h = 10^{5}$&4.9544e-015&1.1102e-016&4.1668e+018&1.3890e+027&0.3181\\
\hline\noalign{\smallskip}
\end{tabular}
\end{table}
From Table \ref{tabl1}, where IEDS and EEDS  stand for Implicit and Explicit Exact Difference Schemes, respectively, we see that the methods of higher order of accuracy have small errors when $T$ and grid size $h$ are small. But when $T$ is large, despite small grid size $h$, the accuracy of these methods decreases due to the accumulation of rounding errors after a large number of iterations. This occurs because for computing the approximate value of the exact solution at the final time $T$ it is needed to compute consecutively the approximate values of the exact solution at every time before $T$. For example, for the fourth order Runge-Kutta method, theoretically, in order to obtain the approximate value of the solution at $T=1$ with the accuracy $10^{-16}$ we have to choose the grid size $h = 10^{-4}$, and perform $10.000$ iterations. Nevertheless, in practice, the actual accuracy reached is only $10^{-14}$ due to the decrease of accuracy as the result of accumulation of rounding errors. Similar situation also occurs with the Taylor methods and other methods of higher order accuracy. Meanwhile, for EDS, if taking grid size $h = T=1$ then after exactly one step we obtain the approximate solution at the time $t_N = T$ with the accuracy $10^{-16}$. This completely agrees with the analysis in Section 5. 

Besides, from Table \ref{tabl1} it is easily seen that for other values of $T \ge 10$ the accuracy of EDS depends slightly on grid sizes $h$ and it is best if $h=T$. In nature, it depends on the rounding errors of computation of the parameters of the schemes and the number of iterations. Meanwhile, from the table we also see that for large time interval the higher order methods are inefficient, even are impossible.\par 
From the above example we can conclude that for the problems on large time intervals the exact difference schemes (implicit or explicit) are more efficient than  higher order methods.

% Đối với các EDS nếu ta lấy bước lưới $h = T$ thì sau đúng 1 bước lặp ta nhận được giá trị xấp xỉ tại thời điểm $t_N = T$ với độ chính xác là $10^{-16}$. Điều này hoàn toàn phù hợp với các phân tích của chúng tôi trong Section 5. Ngoài ra từ Bảng \ref{tabl1} có thể thấy rằng độ chính xác của xác EDS không bị ảnh hưởng nhiều vào bước lưới $h$ mà chỉ phụ thuộc chủ yếu vào số lần lặp và sai số quy tròn khi tính gần đúng các tham số. Vì thế đối với các bài toán trên khoảng thời gian dài có thể kết luận rằng các EDS có ưu thế hơn rất nhiều so với các phương pháp có cấp chính xác cao.
% % Example 4
\begin{example}(Stiff problem)\\
Consider the system $\textbf{x}' = A\textbf{x}$, $t \in [0, T]$ with the coefficient matrix
\begin{equation*}
A = 
\begin{pmatrix}
-1&0&0\\
0& -2& 0\\
0&0& -100\\
\end{pmatrix}.
\end{equation*}
\end{example}
The set of eigenvalues of $A$ is $\sigma(A) = \{-1, -2, -100\}$. For the initial conditions $x(0) = 1, y(0) = 1, z(0) = 1$, the system has the exact solution $x(t) = e^{-t}$, $y(t) = e^{-2t}$, $z(t) = e^{-2017t}$. Obviously, all the components of the solution monotonically tend to zero with the exponential rate. As is well known, for efficiently solving stiff problems it is necessary to use methods having stability properties such as A-stability or L-stability (see \cite{Ascher, Hairer1, Hairer2}. Some implicit Runge-Kutta methods possess these stability propoerties, while explicit Runge-Kutta methods cannot have L-stability since they have bounded stability regions. In general for stiff problems explicit methods are inefficient. \\
% (hàm ổn định là đa thức).
In this example we compare EDS with the classical four-stage Runge-Kutta method, five-order Taylor method and five-order Radau IIA method \cite[Table II.7.7]{Hairer1}. The results of computation are reported in Table \ref{tabl3}, where

% Trong ví dụ này chúng tôi so sánh các EDS với các phương pháp Runge-Kutta bốn nấc kinh điển, phương pháp Taylor chính xác cấp 5 và Radau IIA of order 5 \cite[Table II.7.7]{Hairer1}. Kết quả được thể hiện trong Bảng \ref{tabl3} trong đó
 
\begin{equation*}
error = \max_k\Big\{{|x(t_k) - x_k|} + {|y(t_k) - y_k|} +  {|z(t_k) - z_k|}\Big\} 
\end{equation*}
is a measure of accuracy of methods.
\begin{table}
\caption{ Error of the methods}
\label{tabl3}       % Give a unique label
\begin{tabular}{ccccccc}
\hline\noalign{\smallskip}
$T$&$h$&IEDS error&EEDS error&RK4 error&Taylor error&Radau IIA error\\
\hline\noalign{\smallskip}
$10^{-3}$&$10^{-6}$&6.6613e-016&6.6613e-016&2.4425e-015&2.7756e-014&2.4425e-015\\
&$10^{-5}$&5.5511e-016&6.6613e-016&1.1102e-015&8.5487e-015&6.6613e-016\\
&$10^{-4}$&5.5511e-016&3.3307e-016&7.6037e-012&5.5511e-016& 1.5543e-015\\
&$10^{-3}$&2.2204e-016& 2.2204e-016&8.1964e-008& 1.9596e-011&1.2359e-010\\
\hline\noalign{\smallskip}
$10^{-2}$&$10^{-6}$&1.1102e-015&6.6613e-016&8.6597e-015&2.6584e-013&8.6597e-015\\
&$10^{-5}$&1.0547e-015&6.6613e-016&6.6613e-015&3.8858e-014&3.3307e-015\\
&$10^{-4}$&8.8818e-016&4.9960e-016&3.0913e-011& 3.0531e-015&5.7732e-015\\
&$10^{-3}$&6.1062e-016&4.9960e-016&3.3324e-007&7.9673e-011& 5.0249e-010\\
&$10^{-2}$&2.7756e-016&2.2204e-016&0.0071&1.7611e-004& 4.5087e-005\\
\hline\noalign{\smallskip}
$10^{-1}$&$10^{-6}$&8.5165e-015&2.9616e-015&3.4529e-014&2.2968e-012&3.4418e-014\\
&$10^{-5}$&7.8753e-015&2.6691e-015&1.1999e-014&6.8204e-014&1.1991e-014\\
&$10^{-4}$&7.2122e-015&2.7515e-015&3.0913e-011&9.2176e-015&5.7732e-015\\
&$10^{-3}$&4.0069e-015&1.8644e-015&3.3324e-007&7.9673e-011&5.0249e-010\\
&$10^{-2}$&3.6078e-015& 1.2257e-015&0.0071& 1.7611e-004&4.5087e-005\\
&$10^{-1}$&3.6078e-015&4.3819e-016&291.0000&846.5555&0.0517\\
\hline\noalign{\smallskip}
% $1$&$10^{-6}$&&&2.7933e-011&1.0348e-011&3.4418e-014\\
1&$10^{-5}$&4.5214e-014&7.6050e-015&1.3323e-014&2.4566e-013&1.3212e-014\\
&$10^{-4}$&4.1633e-014&7.3841e-015&3.0913e-011& 2.3925e-014&5.7732e-015\\
&$10^{-3}$&4.1633e-014&7.2164e-015&3.3324e-007&7.9673e-011&5.0249e-010\\
&$10^{-2}$&2.3564e-014&4.7699e-015&0.0071&1.7611e-004&4.5087e-005\\
&$10^{-1}$&1.6376e-014&3.7192e-015& 4.3544e+024&1.8904e+029&0.0517\\
&$1$&5.2180e-015&1.1102e-016& 4.0049e+006&1.3096e+009&0.0264\\
\hline\noalign{\smallskip}
\end{tabular}
\end{table}
Analogously as  Example 3, from Table 
 \ref{tabl3} we see that EDS are much more efficient than RK4 and Taylor methods. Although 
    Radau IIA gives errors better than RK4 and Taylor but it is implicit method, therefore,  it requires more computational cost due to the multiplication of matrix by vector for determining stages of the method. In general, the computational cost of higher order methods is much more than one of EDS.
%    
%    cho sai số tốt hơn phương pháp RK4 và Taylor nhưng đây là phương pháp ẩn, mặc dù dễ dàng đưa phương pháp về dạng hiển nhưng tại mỗi bước lặp cần thực hiện các phép nhân ma trận với vector để xác định các nấc của phương pháp. Nhìn chung các phương pháp chính xác cao (ẩn và hiển) khối lượng tính toán cần thực hiện lớn hơn rất nhiều so với các EDS.
% Example 5
\begin{example}(Nonstandard finite difference scheme of combined type for quasi-nonlinear system of differential equations)\\
%{Lược đồ sai phân khác thường kiểu kết hợp cho một lớp hệ phương trình tựa tuyến tính dựa trên lược đồ sai phân khác thường}\\
Consider the quasi-nonlinear system of equations (see Problem 22.9 \cite{Agarwal1})
\begin{equation}\label{eq:quasilinear}
v' = Av + g(t, v),
\end{equation}
where $g \in C\Big[[t_0, \infty] \times \mathbb{R}^n, \mathbb{R}^n\Big]$ and $||g(t, v)|| \leq \lambda(t)||v||$, where $\lambda(t)$ is a nonnegative continuous function in $[x_0, \infty]$. Additionally, suppose that the function  $\lambda(t)$ satisfies the condition $\int^\infty\lambda(t)dt < \infty$, $\lambda(t) \to 0$ as $t \to \infty$ and the matrix $A$ satisfies $\lambda < 0$ for any $\lambda \in \sigma(A)$. Then it is easy to prove that
 \cite{Agarwal1} every solution of \eqref{eq:quasilinear} is bounded and the trivial solution is asymptotically stable.\par
 Our objective is to construct difference scheme preserving the properties of  \eqref{eq:quasilinear} for any grid size $h > 0$. It should be emphasized that standard finite difference schemes cannot preserve properties of differential equations for any  $h > 0$ \cite{mickens1, mickens2, mickens4}. In this example the properties of the problem are decided by the matrix  $A$, therefore, in the simplest way we propose NSFD scheme for 
 \eqref{eq:quasilinear} in the form
\begin{equation}\label{eq:quasiliner1}
\dfrac{v_{k + 1} - v_k}{\phi} = U(A, v_k, h) + f(t_k, v_k),
\end{equation}
where $U(A, v_k, h)$ is determined so that the scheme
 $\dfrac{v_{k + 1} - v_k}{\phi} = U(A, v_k, h)$ is exact for the system $v' = Av$. Then, by Theorem 5.3.1 in \cite{Agarwal2} the properties of  the problem are preserved for any  $h > 0$. Of course, for ensuring the accuracy of the scheme it is needed to choose   $h < < 1$.
%   tuy nhiên ở đây tính chất của bài toán được bảo toàn với mọi $h > 0$.\par
The scheme \eqref{eq:quasiliner1} has only first order of accuracy, nevertheless it prompts us of a way for constructing NSFD schemes of higher order of accuracy for  \eqref{eq:quasilinear}  in the form
\begin{equation}\label{eq:quasiliner1}
\dfrac{v_{k + 1} - v_k}{\phi} = U(A, v_k, h) + V(f, v_k, t_k, h),
\end{equation}
where $U(A, v_k, h), V(f, v_k, h)$ is determined so that the scheme $\dfrac{v_{k + 1} - v_k}{\phi} = U(A, v_k, h)$ and the scheme $\dfrac{v_{k + 1} - v_k}{\phi} = V(f, v_k, h)$ consecutively are EDS for $v' = Av$ and a scheme of higher order of accuracy for $v' = f(t, v)$. In near future we will develop this idea for constructing NSFD scheme of higher order of accuracy preserving the properties of the general quasi-linear system of differential equations.
\end{example}

\section{Conclusion}
% In this paper we have constructed exact difference schemes for three-dimensional linear system with constant coefficients. This can be considered as a further development of the Roeger's method for two-dimensional linear system, where two parameters $\phi , \theta$ are used. Here we add the third parameter $\psi$ into the discretization of derivative. The obtained results are applicable for systems with the coefficient matrices having any Jordan form matrices. The performed numerical experiments on all cases of the coefficient matrix confirm the obtained theoretical results.\par
% 
% In the future we shall develop the method to system of four linear differential equations with constant coefficients by adding fourth parameter in discretization of derivative, namely, ${dy}/{dt} = -\lim_{h \to 0}\dfrac{y(t + \sigma(h)) - \psi(h)y(t)}{\phi(h)}$, where $\sigma(h) = h + \mathcal{O}(h^2)$. The construction of exact difference schemes for quasilinear system
% $\textbf{x}'(t)$ $ = A(t)\textbf{x}(t) + f(t, \textbf{x})$ also will be our interest. 

In this paper, based on the technique of Mickens and Roeger of exact difference schemes, we have constructed implicit and explicit exact difference schemes (EDS) for system of three linear differential equations with constant coefficients $\textbf{x}' = A\textbf{x}$. We have done the perturbation analysis for showing the advantage of EDS over higher order methods when solving problems on large time intervals. Numerical experiments for several problems, especially, a stiff problem and a periodic problem on large time intervals confirm the advantages of constructed EDS over higher order accuracy methods. In the future we will extend the obtained results to general linear system of $n$ equations with constant coefficients and for constructing NSFD scheme of higher order of accuracy preserving the properties of the general quasi-linear system of differential equations $x' = Ax + f(x)$.
\section*{Acknowledgments}
We would like to thank the reviewers for their helpful comments and suggestions for improving the quality of the paper.\\
This work is supported by Vietnam National Foundation for Science and Technology
Development (NAFOSTED) under the grant  number 102.01-2014.20.

\end{document}